\documentclass[12pt,amscd]{amsart}

\newtheorem{thm}{Theorem}[section]
\usepackage[all]{xy}
\usepackage{graphicx}
\usepackage{amsmath,amsxtra,amssymb,latexsym, amscd,amsthm}
\usepackage{indentfirst}
\usepackage[mathscr]{eucal}
\usepackage{hyperref}
\hypersetup{colorlinks=true,linkcolor=red,filecolor=magenta,urlcolor=green}

\newtheorem{lem}[thm]{Lemma}

\theoremstyle{definition}
\newtheorem{defn}[thm]{Definition}
\newtheorem{exam}[thm]{Example} 
\newtheorem{rem}[thm]{Remark}

\DeclareMathOperator{\Z}{\mathbb {Z}}

\DeclareMathOperator{\depth}{depth}

\DeclareMathOperator{\reg}{reg}
\DeclareMathOperator{\pd}{pd}
\DeclareMathOperator{\dst}{dstab}

\def\x {\mathbf x}

\def\mi {\mathfrak m}

\begin{document}

\title[The Depth  function of path graphs]{The depth  function of powers of cover ideals of path graphs}

\author[T. D.  Dung]{Tran Duc Dung}
	\address{TNU-Thai Nguyen University of Sciences, Phan Dinh phung Ward, Thai Nguyen, Vietnam}
	\email{dungtd@tnus.edu.vn}

\author[N.T. Hang]{Nguyen Thu Hang }
\address{TNU-Thai Nguyen University of Sciences, Phan Dinh phung Ward, Thai Nguyen, Vietnam}
\email{hangnt@tnus.edu.vn}

\author[P. H. Nam]{Pham Hong Nam}
\address{TNU-Thai Nguyen University of Sciences, Phan Dinh phung Ward, Thai Nguyen, Vietnam}
\email{namph@tnus.edu.vn}

\author[N. T. T. Tam]{Nguyen Thi Thanh Tam}
\address{Hung Vuong University, Phu Tho, Vietnam}
\email{nguyenthithanhtam@hvu.edu.vn}

\subjclass{13A15, 13C15, 05C90, 13D45.}
\keywords{Depth function, Cover ideal, Induced  matching, Ordered matching, Path}

\begin{abstract} Let $G=P_n$ be a path graph with cover ideal $J(P_n)$. By using Hochster's depth formula, we prove the explicit formulae to compute the depth functions of powers of cover ideals of paths.
\end{abstract}

\maketitle
\section{Introduction}

Let $R= K[x_1,\ldots,x_n]$ be a polynomial ring over a field $K$ and  $I$ the  homogeneous ideal of $R$. It is a classical result of Brodmann that the sequence Brodmann $\depth R/I^t$ of non-negative integers stabilizes for sufficiently large $t$ i.e., it eventually becomes constant as $t\rightarrow \infty$ \cite{Br}. Understanding the asymptotic behavior of this sequence, as well as determining its initial values for specific classes of ideals, has become an active area of research at the interface of commutative algebra and combinatorics.
In this direction, Ha, Nguyen, Trung, and Trung \cite{HHTT}  proved that the depth functions of powers of monomial ideals can realize any non-negative integer-valued convergent function. The smallest positive natural number $k$ such that 
$$\depth R/I^m={\rm lim}_{t\rightarrow \infty}\depth R/I^t \text{ for all } m\geqslant k$$
 is called the {\it index of depth stability of powers} of $I$ and is denoted by $\dst(I) $. Despite intensive study, even for general monomial ideals the best known bounds for $\dst(I)$ are exponential in the number of variables and the maximal degree of generators \cite{HT1, HT2}. \\ 
For special monomial ideals arising from graphs, deeper results are known. In particular, when $I$ is the edge ideal of a graph or the cover ideal of a graph, explicit formulas for $\dst(I)$  have been established by Lam, Trung, Trung \cite{LTT},  and Binh, Hang, Hien, Trung \cite{BHHT}.  Moreover, depth and symbolic depth of both edge ideals and cover ideals have been studied extensively in recent literature, including bounds and exact formulas in terms of combinatorial invariants such as matching numbers,  induced matching numbers, and ordered matching numbers.

While the index $\dst(I)$ has been studied broadly, computing explicit values of $\depth R/I^t$ for small powers $t$ remains difficult. Until now, only a limited number of results are known for edge ideals of specific graph families: paths \cite{BC}, cycles and trees \cite{MTV}, Cohen-Macaulay tree graphs \cite{HHT}, etc. In contrast, there are relatively few results giving explicit formulas for $\depth R/J^t$, where $J=J(G)$ is the cover ideal of a particular simple graph.

In this paper,  we investigate the depth functions and depth stability indices of square-free monomial ideals of height two, which coincide with cover ideals of graphs.  Let $G=(V,E)$ be a graph with vertex set $V=\{1,\ldots,n\}$. For a subset $\tau$ of $V$, denote the square-free monomial $\x_{\tau}$ to be the product of all variables $x_i$ where $i\in\tau$. Then, the {\it cover ideal} of $G$ is defined by:
$$J(G) = (\x_{\tau} \mid \tau \text{ is a minimal vertex cover of } G).$$

Our first main result investigates the behavior of $\depth R/J(P_n)^t$ for even paths, establishing a closed formula in terms of $n$ and $t$. More precisely, for $n = 2k$, we prove  that

{\bf Theorem 3.3.} {\it Let $P_n$  be an even path with  $n=2k$ and $k\geqslant 1$.  Then
 $$\depth R/J(P_{2k})^t = \begin{cases}
 k-1+ \left\lfloor  \dfrac{k+t}{2t+1}\right\rfloor & \text{ for all }  1\leqslant t\leqslant k -1,\\
 k-1  & \text{ for all } t\geqslant  k.
 \end{cases}$$}
 
 The next purpose of this paper is to complete the picture for path graphs by treating the case of odd paths. Our the second main result provides explicit formulas for $\depth R/J(P_n)^t$ when $n$ is odd,

 {\bf Theorem 3.4.} {\it Let $P_n$  be an odd path with cover ideal $J(P_n)$. 
\begin{itemize}
\item[(i)] If  $n=4k+3$ with $k\geqslant 0$, then
$$\depth R/J(P_{4k+3})^t =
 \begin{cases}
(2k+1)+ \left\lfloor\dfrac{2k+1}{2t+1}\right\rfloor& \text{ for all }  1\leqslant t\leqslant k,\\
 2k+1  & \text{ for all } t\geqslant  k+1.
 \end{cases}$$
  \item[(ii)]  If  $n=4k+1$ with $k\geqslant 1$, then
 \begin{equation*}
\depth R/J(P_{4k+1})^t=
 \begin{cases}
2k+ \left\lfloor \dfrac{2k}{2t+1}\right\rfloor &\text{  for all  }1\leqslant t\leqslant k -1,\\
2k  &\text{  for all  }t\geqslant  k.
 \end{cases}
 \end{equation*}
 \end{itemize}}

The proof strategy relies on Hochster's depth formula \cite{HOC}. For a monomial ideal $I$, the set of associated radical ideals ${\rm assrad(I)}$ consists of radicals $\sqrt{I:f}$ as $f$ ranges over monomials not in $I$. Then, Hochster's formula states that.
$$\depth R/I=\min \{\depth R/Q |\ Q \text{ is an associated radical of }I \}.
$$
Applied to $I = J(P_n)^t$, this reduces the computation to determining depths of certain radical ideals of the form $\bigcap_{e_i\in S}(x_i,x_{i+1})$, where $S \subseteq E(P_n)$ is a set of edges in $P_n$ arising from a monomial $f \notin J(P_n)^t$. Geometrically, each such radical ideal corresponds to a forest $G_S$ whose edges are exactly $e_i=\{x_i,x_{i+1}\}$ for $e_i\in S$. For a forest, the depth can be expressed in terms of its induced matching number $\nu'(G_S)$ via the formula $\depth R/J(G_S) = n - \nu'(G_S) - 1$. The key combinatorial insight is that the condition $f \notin J(P_n)^t$ imposes constraints on the structure of $S$, which in turn limits how large $\nu'(G_S)$ can be. By constructing optimal configurations using a block decomposition of the path into segments of length $2t+1$, we show that the maximum possible induced matching number of $G_S.$ By the careful algebraic manipulation, the explicit formulas stated in our paper.

 \rm The paper is organized as follows. Section 2 collects necessary preliminaries on graphs, the depth function, Hochster’s depth formula, and induced matchings of certain graphs. In Section 3, we present the detailed computation of depth for even and odd paths.

\section{Preliminary}

\rm In this section, we recall notation, terminology and basic results used in the paper. Throughout the paper, let $K$ be a field, and $R = K[x_1,\ldots,x_n]$, with $n\geqslant 2$ be a polynomial ring, and let $\mi = (x_1,\ldots,x_n)$ be the {\it maximal homogeneous} ideal of $R$. 

 \subsection{Depth and Caselnuovo Mumford regularity} The main object of our work is the depth of graded modules and ideals over $R$. This invariant can be defined via either the minimal free resolutions or the local cohomology modules. 

Let $L$ be a nonzero finitely generated graded  $R$-module and let
$$0 \rightarrow \bigoplus_{j\in\Z} R(-j)^{\beta_{p,j}(L)} \rightarrow \cdots \rightarrow \bigoplus_{j\in\Z}R(-j)^{\beta_{0,j}(L)}\rightarrow 0$$
be the minimal free resolution of $L$. The {\it projective dimension} of $L$ is the length of this resolution
$$\pd(L) = p,$$
and the {\it depth} of $L$ is given by Auslander-Buchsbaum formula
\begin{equation}
\depth(L) =n- p.
\end{equation}

Another invariant measures the complexity of the resolution is the \emph{Castelnuovo–Mumford regularity} (or regularity for short) of $L$ which is defined by
$$\reg(L) = \max\{j-i\mid \beta_{i,j}(L)\ne 0\}.$$

The depth and regularity of $L$ can also be computed via the local cohomology modules of $L$. Let $H_{\mi}^i(L)$ be the $i$-th cohomology module of $L$ with support in $\mi$. Then,
$$\depth(L) = \min\{i\mid H_{\mi}^i(L) \ne 0\},$$
and
$$\reg(L) = \max\{j + i\mid H_{\mi}^i(L)_j \ne 0, \text{ for } i = 0,\ldots, \dim(L), \text{ and } j\in \Z\}.$$

 Let $I$ be a monomial ideal in $R$, from \cite{HOC} we recall the concept {\it associated radical ideal} of $I$, as follows.
\begin{defn} Let $I$ be a monomial ideal in $R$ and $u$ be a monomial, which is not in $I$. The radical ideal $Q:=\sqrt{I:u}$ is called an associated radical ideal of $I$. We denote the set of all {\it associated radical} ideals of I by ${\rm assrad(I)}.$
\end{defn} 

\begin{rem}\label{associatedradical} Let $I$ be a monomial ideal in $R$. Then $I$ admits a unique standard primary decomposition
  $$I=\bigcap_{i=1}^r Q_i,$$ 
  where each $Q_i$ is a $P_i$-primary monomial ideal. Given such decomposition of $I$, for  each monomial $u\in R$, we have 
  $$\sqrt{I:u}=\underset{u\notin Q_i}\bigcap P_i.$$
   In fact, each associated radical ideal of $I$ is an intersection of some associated prime ideals of $I$. Moreover, because $\sqrt{I}=\sqrt{I:1}$, it follows that all associated prime ideals $P_i$ belong to ${\rm assrad(I)}.$
\end{rem}

\begin{exam} Let $I=(x_1,x_2^2)\cap (x_2,x_3^3)\cap (x_1,x_3)$ and let $P_1=(x_1,x_2), P_2=(x_2,x_3), P_3=(x_1,x_3)$. Then ${\rm assrad(I)}=\{P_1, P_2, P_3, P_1\cap P_2, P_2\cap P_3, P_1\cap P_2\cap P_3\}.$

\end{exam}

We employ the set of associated radical ideals in order to investigate the depth of $R/I$. The significance of this set is highlighted by the {\it Hochster's depth formula} given in \cite{HOC}, which states that:
\begin{equation}\label{Hochsterformula}
\depth R/I=\min \{\depth R/Q |\ Q \text{ is an associated radical of }I \}.
\end{equation}
 This formula provides a powerful tool for estimating or computing the depth of a monomial ideal via the depths of its associated radical ideals.

\subsection{Graphs}

Let $G$ be a simple graph. We use the symbols $V(G)$ and $E(G)$ to denote the vertex set and the edge set of $G$, respectively. In this paper, we always assume that $E(G)\ne \emptyset$ unless otherwise indicated.

A graph $H$ is called a {\it subgraph} of $G$ if $V(H)\subseteq V(G)$ and $E(H) \subseteq E(G)$.  A graph $H$ is called an {\it induced subgraph} of $G$ if the vertices of $H$ are vertices in $G$, and for vertices $u$ and $v$ in $V(H)$, $\{u,v\}$ is an edge of $H$ if and only if $\{u,v\}$ is an edge in $G$. The induced subgraph of $G$ on a subset $S \subseteq V(G)$, denoted by $G[S]$, is obtained by deleting vertices not in $S$ from $G$ (and their incident edges).

Let $P\colon ,v_1,\ldots, v_k$ be a sequence of vertices of $G$. Then,
\begin{enumerate}
\item $P$ is called a {\it walk} if $\{v_{i},v_{i+1}\}\in E(G)$ for $i=1, \ldots,k$. In this case, we say that $p$ is a walk from $v_1$ to $v_k$.
\item $P$ is called a {\it path} if it is a walk and every vertex appears exactly once.
\end{enumerate}
In each case, denote $k=\text{length}(P)$ which is called the {\it length} of $P$. 

A graph $G$ with $n$ vertices such that all edges lying on a path is called a path with $n$ vertices, denoted by $P_n$.

A  graph is {\it connected} if there is a path from any vertex to any other vertex in the graph. A graph that is not connected is said to be {\it disconnected}. A connected component of a graph $G$ is a connected subgraph that is not part of any larger connected subgraph. A connected graph without cycles is a {\it tree}. A graph is a {\it forest} if every its connected component is a tree.

The graph $G$ is {\it bipartite} if  $V(G)$ can be partitioned into two subsets $X$ and $Y$ such that every edge has one end in $X$ and another end in $Y$; such a partition $(X,Y)$ is called a {\it bipartition} of the graph. Note that $G$ is bipartite if and only if it has no cycle of odd length (see \cite[Theorem 4.7]{BM}).  So, any path $P_n$ is bipartite.

A {\it matching} in the graph $G$ is a set of pairwise non adjacent edges. If $M$ is a matching, the two ends of each edge of $M$ are said to be matched under $M$, and each vertex incident with an edge of $M$ is said to be covered by $M$.  The number of edges in a maximum matching in a graph $G$ is called the {\it matching number} of $G$ and denoted $\nu(G)$.

A matching $M$ of $G$ is called an {\it induced matching} if the graph $G[M]$ is just disjoint edges. The {\it induced matching number} of $G$, denoted by $\nu'(G)$, is the maximum size of an induced matching in $G$.

\begin{lem}\label{inducedmatching} {\rm \cite[Remark 2.12]{BHT} } Let  $P_n$ be a path. Then  $\nu'(P_n)=\left\lfloor  \dfrac{n+1}{3}\right\rfloor$.
\end{lem}

An {\it independent set} in $G$ is a set of vertices no two of which are adjacent to each other.  According to Constantinescu and Varbaro \cite{CV}, we define an ordered matching as follows.

\begin{defn}\label{ordered-matching} A matching $M=\{\{u_i,v_i\} \mid i=1,\ldots,s\}$ in a graph $G$ is called an {\it ordered matching} if:
\begin{enumerate}
\item $\{u_1,\ldots,u_s\}$ is an independent set in $G$,
\item $\{u_i, v_j\} \in E(G)$ implies $i \leqslant j$.
\end{enumerate}

The {\it ordered matching number} of $G$, denoted by $\nu_0(G)$ is the maximum size of an ordered matching in $G$. 
\end{defn}

\begin{lem}\label{orderedmat} {\rm \cite[Proposition 3.5]{BHHT} } Let  $P_n$ be a path. Then  $\nu_0(P_n)=\left\lfloor  \dfrac{n}{2}\right\rfloor$.
\end{lem} 

\subsection{Edge Ideals and Cover Ideals} 
Let $G$ be a finite simple graph. Assume that $V(G)=\{1,\ldots,n\}$. The {\it edge ideal} of $G$ is define by
$$I(G): = (x_ix_j \mid \{i,j\} \in E(G)) \subseteq R.$$

 A {\it vertex cover} of $G$ is a subset of $V$ which meets every edge of $G$; a vertex cover is {\it minimal} if none of its proper subsets is itself a cover. The {\it cover ideal} of $G$ is defined by 
$$J(G) := (\x_{\tau} \mid \tau \text{ is a minimal vertex cover of } G).$$
It is well-known that the cover ideal $J(G)$ has the primary decomposition
\begin{equation} \label{intersect}
J(G) = \bigcap_{\{u,v\}\in E} (x_u, x_v).
\end{equation}
It follows that $J(G) = I(G)^{\vee}$.

The $s$-th symbolic power of $J(G)$ is 

\begin{equation}\label{symbolic}
J(G)^{(s)}=\underset{\{i,j\}\in E(G)}\bigcap (x_i,x_j)^s.
 \end{equation}

It is worth mentioning that the cover ideal $J(G)$ of $G$ is normally torsion-free, i.e. $J(G)^{(s)} = J(G)^s$ for all $s\geqslant 1$, if and only if $G$ is bipartite (see \cite[Theorem 5.1]{HHT1}).   In particular, $J(P_n)^{(s)} = J(P_n)^s$ for all $s\geqslant 1$.

\subsection{Depth fucntion of cover ideals of graphs} In the sequel, we need some facts about the behavior of the depth function of $J(G)$. By  \cite[Theorems 3.2 and 3.4]{HKTT}, it follows.

\begin{lem}\label{hktt} Let $G$ be a simple graph with cover ideal $J(G)$. Then,
\begin{enumerate}
\item[(i)] The sequence $\{\depth R/J(G)^{(s)}\}_{t\geqslant 1}$ is non-increasing, i.e.
$$\depth R/J(G) \geqslant \depth R/J(G)^{(2)} \geqslant \depth R/J(G)^{(3)} \geqslant \cdots$$
\item[(ii)] $\depth R/J(G)^{(s)} = n - \nu_0(G)-1$ for all $s\geqslant 2\nu_0(G)-1$.
\end{enumerate}
\end{lem}
As a consequence for bipartite graph, we obtain.
\begin{lem}\label{sdstability} Let $G$ be a bipartite graph. Then,
$$\dst(J(G)) = \min \{t\geqslant 1\mid \depth R/J(G)^{t} \leqslant  n-\nu_0(G)-1\}.$$
\end{lem}
According to Lemma \ref{sdstability}, we derive the following result in \cite[Proposition 3.5]{BHHT}.
\begin{lem}\label{pathsdstability} Let $P_r$ be the path with $r$ vertices. Then,
$$
\dst(J(P_r))=
\begin{cases}
\frac{r}{2} & \text{ if } r \text{ is even},\\
\left\lceil \frac{r-1}{4}\right\rceil & \text{ if } r \text{ is odd}.
\end{cases}
$$    
\end{lem}

By Lemma  \cite[Theorem 5.59]{MS}, and the fact $J(G) = I(G)^{\vee}$, it follows. 
\begin{lem}\label{equals} $\pd R/J(G) = \reg I(G).$
\end{lem}
By  Lemma \ref{equals}, and together with \cite[Theorem 1.3.3]{BH}, it yields. 
\begin{lem}\label{depthreg1} Let $G$ be a simple graph. Then, $\depth R/J(G)=n-\reg(I(G)).$
\end{lem}

\section{Depth of powers of cover ideals of paths}
In this section,  we determine the depth of the powers of cover ideals associated with paths. For a real number $x$, denote  $\left\lceil x\right\rceil$ the least integer at least $x$,   $\left\lfloor x\right\rfloor$ the largest integer at most $x$. First, we have a simple lemma.
\begin{lem}\label{phannguyen}  Let $a, b$ be positive integers, and $x$ be a real number. Then 
\begin{itemize}
\item[(i)] $\left\lceil \dfrac{a}{2}\right\rceil=\left\lfloor \dfrac{a+1}{2}\right\rfloor.$
\item[(ii)] $\left\lfloor \dfrac{\left\lfloor x\right\rfloor}{b}\right\rfloor=\left\lfloor \dfrac{x}{b}\right\rfloor.$
\end{itemize}
\end{lem}
\begin{proof}  $(i)$ By considering $a$ is even and $a$ is odd. The statement can be easily derived from the definition of the integer part function.\\
$(ii)$ Suppose $\left\lfloor \dfrac{x}{b}\right\rfloor=k$. Then, $k\leqslant \dfrac{x}{b}<k+1,$ that means $bk\leqslant x <b(k+1).$ As $bk$  is an integer less than or equal to $x$, we have $bk \leqslant \left\lfloor x\right\rfloor$. Therefore, 
$$bk \leqslant \left\lfloor x\right\rfloor \leqslant x<b(k+1).$$
By this, it follows $k\leqslant \left\lfloor \dfrac{\left\lfloor x\right\rfloor}{b}\right\rfloor<k+1$. Thus, we have $\left\lfloor \dfrac{\left\lfloor x\right\rfloor}{b}\right\rfloor=k,$ as required. 
 \end{proof}
Let $P_n$ be a path graph with vertex set $V(P_n)=\{x_1, \ldots, x_n\}$, and edge set \break $E(P_n)=\big\{e_i=\{x_i,x_{i+1}\}\mid i=1, \ldots, n-1\big\}$ with $E(P_n)\neq \emptyset$. By \eqref{intersect}, the cover ideal of $P_n$ as follows
$$J(P_n)= \bigcap_{i=1}^{n-1}e_i=\bigcap_{i=1}^{n-1} (x_{i}, x_{i+1}),$$
and $$J(P_n)^t= \bigcap_{i=1}^{n-1}e_i^t=\bigcap_{i=1}^{n-1} (x_{i}, x_{i+1})^t, \text{ for all } t\geqslant 1.$$
For each monnomial $f=x_1^{a_1}\ldots x_n^{a_n}\notin J(P_n)^t$, by Remark \ref{associatedradical}, it follows
\begin{equation}
\sqrt{J(P_n)^t\colon f}=\bigcap\limits_{ e_i\in S}e_i =\bigcap\limits_{ e_i\in S} (x_{i}, x_{i+1}),
\end{equation}
where 
$$S=\{e_i\mid f\notin (x_i,x_{i+1})^t\}=\{e_i \mid a_i+a_{i+1}\leqslant t-1\}.$$
Let $J_S=\bigcap\limits_{e_i\in S} (x_{i}, x_{i+1})$, and we set $G_S$ is a subgraph of $P_n$, in which $E(G_S)=\big\{e_i=\{x_i, x_{i+1}\}\mid e_i\in S\big\}.$ Obvisously, $G_S$ is a forest with its each connected component is a path. By Hochster's depth formula \eqref{Hochsterformula} we can compute 
\begin{align}
\depth R/J(P_n)^t&=\underset{f\notin J(P_n)^t}\min\{\depth R/\sqrt{J(P_n)^t\colon f}\}\notag\\
&=\underset{S}\min \{\depth R/J_S\},
\end{align}
and the minimal value takes over for all $S,$ which can appear from a monomial $f\notin J(P_n)^t.$ 

For each set $S$, by the Lemma \ref{depthreg1}, it follows that 
\begin{equation}\label{depthreg2}
\depth R/J_S=n-\reg I(G_S),
\end{equation}
in which $I(G_S)$ is edge ideal of graph $G_S.$ Thus, 

\begin{equation}\label{depthreg3}
\depth R/J(P_n)^t=\underset{S}\min \{\depth R/J_S\}=n-\underset{S}\max\reg I(G_S),
\end{equation}
in which $I(G_S)$ is edge ideal of graph $G_S$, and the maximal value takes over for all $S,$ which can appear from a monomial $f\notin J(P_n)^t.$ 

 In order to compute $\reg I(G_S)$, we recall the following result. 

\begin{lem}\cite[Theorem 4.7 ]{BHT}\label{Hahuytai} Let  $G$ be a forest with edge ideal $I=I(G).$ Let $\nu'(G)$ denote the induced matching number of $G$. Then, 
\begin{equation*}
\reg I(G)=\nu'(G)+1.
\end{equation*}
\end{lem}
 
  The rest of the paper is devoted to compute explicitly the  depth functions for the cover ideals of paths. We start with one of the main results in this section is cover ideals of even paths.
    
  \begin{thm}\label{main1}
Let $P_n$  be an even path with  $n=2k$.  Then
$$\depth R/J(P_{2k})^t = \begin{cases}
k-1+ \left\lfloor  \dfrac{k+t}{2t+1}\right\rfloor & \text{ for all }  1\leqslant t\leqslant k -1,\\
 k-1  & \text{ for all } t\geqslant  k.
 \end{cases}$$
   \end{thm}
    \begin{proof}  For $1\leqslant t\leqslant k -1$, by \eqref{depthreg3} we have 
   $$\depth R/J(P_{2k})^t=\underset{S}\min \{\depth R/J_S\}=2k-\underset{S}\max\reg I(G_S),$$
   in which $I(G_S)$ is edge ideal of the forest $G_S$, and the maximal value takes over for all $S=\big\{e_i \in E(P_{2k})\mid a_i+a_{i+1}\leqslant t-1\big\}$, which can appear from a monomial $f\notin J(P_{2k})^t.$ In order to maximize $\reg I(G_S)$, we need  the forest $G_S$ have to contain as many discrete connected components as possible. By Lemma \ref{Hahuytai}, it follows 
   \begin{align}\label{findmax}
\depth R/J(P_{2k})^t=2k-\max\reg I(G_S)=2k-\underset{S} \max \ \nu'(G_S)-1. 
\end{align}
   Set $\Delta=\underset{S} \max \ \nu'(G_S)$, we need to investigate the $G_S$, for all set $S$, to get $\Delta.$ To do this, we will choose a configuration of  the set $S$ such that it has as many disjoint edges as possible, as follows:
   
    We consider a block consisting of $2t$ edges consecutive of $P_{2k}$. In this block, selected  disjoint edges, for example, edges in odd positions $((1-2), (3-4), \ldots, ((2t-1)-2t))$. Place these edges in $S$, with the following value assignment:
   \begin{align*}
   e_1=\{x_1, x_2\}\in S &\text{ with } a_1=0,  \text{ and } a_2\leqslant t-1;\\
   e_2=\{x_2, x_3\}\notin S &\text{ with } a_2\leqslant t-1, \text{ and }  a_3\geqslant 1;\\
   e_3=\{x_3, x_4\}\in S &\text{ with } a_3\geqslant 1, \text{ and }  a_4\leqslant t-2;\\
    e_4=\{x_4, x_5\}\notin S &\text{ with } \ a_4\leqslant t-2, \text{ and }  a_5 \geqslant 2;\\
     \ldots  \ldots  \ldots  \ldots  \ldots  \ldots & \ldots  \ldots  \ldots  \ldots  \ldots  \ldots  \ldots  \ldots  \ldots 
   \end{align*}

   We observe that after each pair of edges $e_i\in S$ and $e_{i+1}\notin S$, with $i \geqslant 1$, the minimum exponent at the odd vertex increases by $1$ unit. Therefore, after $t$ such pairs of edges are selected, the exponent at the $2t+1$ odd vertex will reach the threshold $t$. This implies that we cannot choose another edge belonging to $S$ while still satisfying the total exponent remains less than or equal to $t-1$. Hence, to connect the two blocks, an edge not belonging to $S$ is needed between the last vertex of the first block (value $t$) and the first vertex of the second block (value $0$). By making such a selection, we create consecutive blocks of $2t+1$ edges of $P_{2k}$, where $t$ edges belong to $S$, interspersed with $t-1$ edges not belonging to $S$, and the last two edges of the block do not belong to $S$. 
   
   Since $1\leqslant t\leqslant k -1$, it follows $2t+1\leqslant 2k-1$, thus by this choosing, we can  get $\left\lfloor  \dfrac{2k-1}{2t+1}\right\rfloor$ blocks of $P_{2k}$. If the number of remaining edges is $r$ with $0 \leqslant r < 2t+1$, we can choose up to $\left\lfloor  \dfrac{r+1}{2}\right\rfloor$ additional disjoint edges in the remainder, with the same value assignment method.
   
   Therefore, we have divided the set edges of $P_{2k}$  into $m=\left\lfloor  \dfrac{2k-1}{2t+1}\right\rfloor$ complete blocks (each block having $2t+1$ edges) and the remainder contains $r$ edges, with:
   \begin{align}\label{eq}
   2k-1=m(2t+1)+r, \text{ where } 0 \leqslant r \leqslant  2t.
   \end{align}
     Then, the following can be achieved $\Delta=\underset{S} \max \ \nu'(G_S)=mt+\left\lfloor  \dfrac{r+1}{2}\right\rfloor.$ Since,  $2k-1=m(2t+1)+r, \text{ with } 0 \leqslant r \leqslant  2t.$ It deduces 
  \begin{align}\label{findk}
k=mt +\dfrac{m+r+1}{2}.
\end{align}
   It follows
   $$\Delta=mt+\left\lfloor  \dfrac{r+1}{2}\right\rfloor=k- \dfrac{m}{2}-\dfrac{r+1}{2}+\left\lfloor  \dfrac{r+1}{2}\right\rfloor.$$
   Set $u=\frac{r+1}{2}-\left\lfloor  \frac{r+1}{2}\right\rfloor$, we have $0\leqslant u<1$, and  then $\Delta=k- \frac{m}{2}-u.$ Since $\Delta$ is an integer, $u$ must compensate for the fractional part of $\frac{m}{2}$.
        We investigate  the following two cases:
   
 {\it Case $1$:} If $m=2p$ is even,  it follows  $u=0$ (it means  that $r$ is odd). Thus $\Delta=k-p.$
 
    {\it Case $2$:} If $m=2p+1$ is odd, then  $\frac{m}{2}=p+\frac{1}{2}$. To keep $\Delta$  as an integer then $u=\frac{1}{2}$ (it means that $r$ is even). Hence $\Delta=k-p-1.$   From Cases $1$ and $2$, we can write
    $$\Delta=k-\left\lfloor  \frac{m+1}{2}\right\rfloor.$$
 In the other hand, by \eqref{eq} we have 
 $$\dfrac{k}{2t+1}=\dfrac{m}{2}+\dfrac{r+1}{2(2t+1)}.$$
 Set $q=\frac{r+1}{2(2t+1)},$ as $0 \leqslant r \leqslant  2t$ we get $q\in [\frac{1}{2(2t+1)}, \frac{1}{2}]$. Therefore, 
 $$\dfrac{k}{2t+1}=\dfrac{m}{2}+q, \text{ with  }q\in \big(0, \frac{1}{2}\big].$$
 Next, we consider 
 $$\dfrac{k+t}{2t+1}=\dfrac{k}{2t+1}+\dfrac{t}{2t+1}=\dfrac{m}{2}+q+\dfrac{t}{2t+1}.$$
 We set again $\tau=q+\dfrac{t}{2t+1}$. As $q\leqslant \frac{1}{2}$, so $\tau< \frac{1}{2}+\frac{1}{2}=1.$ Moreover, because of  $q\geqslant \frac{1}{2(2t+1)}$, so $\tau \geqslant \frac{1}{2}$. Hence, $\tau\in [\frac{1}{2}, 1).$
 Therefore, 
  \begin{align}\label{eq2}
 \dfrac{k+t}{2t+1}=\dfrac{m}{2}+\tau, \text{ with  }\tau \in  \big[\dfrac{1}{2}, 1\big).
 \end{align}
 By \eqref{eq2}, we can compute the number $\left\lfloor  \dfrac{k+t}{2t+1}\right\rfloor$ in terms of the following two cases:
 
  {\it Case $(i)$:} If $m=2p$ is even,  we imply 
  $$\left\lfloor  \dfrac{k+t}{2t+1}\right\rfloor=\left\lfloor p+\tau\right\rfloor=p=\left\lceil \dfrac{m}{2}\right\rceil.$$
 
    {\it Case $(ii)$:} If $m=2p+1$ is odd,  we have 
  $$\left\lfloor  \dfrac{k+t}{2t+1}\right\rfloor=\left\lfloor p+\dfrac{1}{2}+\tau\right\rfloor=p+1, \text { by } \tau\in [\frac{1}{2}, 1).$$
  Therefore, $\left\lfloor  \dfrac{k+t}{2t+1}\right\rfloor=p+1=\left\lceil \dfrac{m}{2}\right\rceil.$ By Lemma  \ref{phannguyen}, we have $\left\lceil \dfrac{m}{2}\right\rceil=\left\lfloor  \dfrac{m+1}{2}\right\rfloor$. Thus, it follows:
   \begin{align}\label{eq3}
   \left\lfloor  \frac{m+1}{2}\right\rfloor=\left\lfloor  \dfrac{k+t}{2t+1}\right\rfloor.
   \end{align}
    
 By this, we deduce
   \begin{align}\label{findM}
 \Delta=k-\left\lfloor  \dfrac{k+t}{2t+1}\right\rfloor.
 \end{align}
 Combining \eqref{findmax} and \eqref{findM}, for $1\leqslant t\leqslant k -1$ we have 
 $$\depth R/J(P_{2k})^t=2k-\bigg(k-\left\lfloor  \dfrac{k+t}{2t+1}\right\rfloor \bigg)-1=k-1+\left\lfloor  \dfrac{k+t}{2t+1}\right\rfloor.$$
 
    For $t\geqslant k$, by Lemma \ref{orderedmat}  and Lemma \ref{pathsdstability} with $n=2k$, it  follows $$\dst(J(P_{2k}))=\frac{n}{2}=k.$$  Moreover, by Lemma \ref{hktt}$(ii)$, for all $t\geqslant k$, we get $$\depth R/J(P_{2k})^{t}=2k-\nu_0(P_{2k})-1=2k-\left\lfloor  \frac{2k}{2}\right\rfloor-1=2k-k-1=k-1.$$
    The conclusion of the theorem follows  
    \end{proof}
    In the last of section, we study the formulae to calculate the depth of paths with an odd number of vertices. In this case, we also have to divide the number of vertices into two separate cases.
     \begin{thm}
Let $P_n$  be an odd path with cover ideal $J(P_n)$.
\begin{itemize}
\item[(i)] If  $n=4k+3$ with $k\geqslant 0$, then
$$\depth R/J(P_{4k+3})^t =
 \begin{cases}
(2k+1)+ \left\lfloor\dfrac{2k+1}{2t+1}\right\rfloor& \text{ for all }  0\leqslant t\leqslant k,\\
 2k+1  & \text{ for all } t\geqslant  k+1.
 \end{cases}$$
  \item[(ii)]  If  $n=4k+1$ with $k\geqslant 1$, then
 \begin{equation*}
\depth R/J(P_{4k+1})^t=
 \begin{cases}
2k+ \left\lfloor \dfrac{2k}{2t+1}\right\rfloor &\text{  for all  }1\leqslant t\leqslant k -1,\\
2k  &\text{  for all  }t\geqslant  k.
 \end{cases}
 \end{equation*}
 \end{itemize}
   \end{thm}

 \begin{proof}  Unlike the proof of Theorem \ref{main1}, in this theorem, we first investigate the depth of $J(P_n)$ when $n$ is odd in the case where stability occurs.  Indeed, by Lemma \ref{pathsdstability}, we have:
 \begin{itemize}
 \item[•]  If $n=4k+3$, then $\dst(J(P_{4k+3}))= \left\lceil \dfrac{4k+3-1}{4}\right\rceil=k+1$, and by Lemma \ref{hktt}$(ii)$ for all $t\geqslant k+1$, we have
 \begin{align*}
 \depth R/J(P_{4k+3})^{t}&=(4k+3)-\nu_0(P_{4k+3})-1=(4k+3)-\left\lfloor  \dfrac{4k+3}{2}\right\rfloor-1\\
 &=4k+3-(2k+1)-1=2k+1.
 \end{align*}

 \item[•]  If $n=4k+1$, then $\dst(J(P_{4k+1}))= \left\lceil \dfrac{4k+1-1}{4}\right\rceil=k$, and by Lemma \ref{hktt}$(ii)$ for all $t\geqslant k$, we have
 \begin{align*}
 \depth R/J(P_{4k+1})^{t}&=(4k+1)-\nu_0(P_{3k+3})-1=(4k+1)-\left\lfloor  \dfrac{4k+1}{2}\right\rfloor-1\\
 &=4k+1-2k-1=2k.
 \end{align*}
 \end{itemize}
 Next, we will demonstrate the case the depth function of $J(P_n)$  is non-increasing. By \eqref{depthreg3} we have 
   $$\depth R/J(P_{n})^t=\underset{S}\min \{\depth R/J_S\}=n-\underset{S}\max\reg I(G_S),$$
   Set $\Delta=\underset{S} \max \ \nu'(G_S)$. The problem reduces to finding the maximum possible value of $M$ given that there exists a sequence of numbers $a_1, \ldots, a_n$ satisfying the constraint from $S=\big\{e_i \in E(P_{n})\mid a_i+a_{i+1}\leqslant t-1\big\}.$ The method for choosing the optimal S is similar to the proof of Theorem \ref{main1}, which is: we choose consecutive blocks of $2t+1$ edges of $P_{n}$, where $t$ edges belong to $S$, interspersed with $t-1$ edges not belonging to $S$, and the last two edges of the block do not belong to $S$, specifically as follows:
     \begin{align*}
   e_1=\{x_1, x_2\}\in S &\text{ with } a_1=0,  \text{ and } a_2\leqslant t-1;\\
   e_2=\{x_2, x_3\}\notin S &\text{ with } a_2\leqslant t-1, \text{ and }  a_3\geqslant 1;\\
   e_3=\{x_3, x_4\}\in S &\text{ with } a_3\geqslant 1, \text{ and }  a_4\leqslant t-2;\\
    e_4=\{x_4, x_5\}\notin S &\text{ with } \ a_4\leqslant t-2, \text{ and }  a_5 \geqslant 2;\\
    \ldots  \ldots  \ldots  \ldots  \ldots & \ldots  \ldots  \ldots  \ldots  \ldots  \ldots  \ldots  \ldots  \ldots  \ldots 
   \end{align*}
   Let $N=n-1$ be the number of edges of $P_n$. Then, divide $N$ into
   \begin{itemize}
   \item[•] The complete blocks number $m=\left\lfloor  \dfrac{N}{2t+1}\right\rfloor;$
   \item[•] The remaimder  edges $r=N-m(2t+1), \text{ in which } 0\leqslant r\leqslant 2t.$
   \end{itemize}
   Then, from the selected configuration, we have 
  \begin{align}\label{findMM}
  \Delta=\underset{S} \max \ \nu'(G_S)=mt+\left\lfloor  \dfrac{r+1}{2}\right\rfloor.
  \end{align}
     As $n$ is odd, it follows $N$ is even. Let $N=2A$, with $A=\dfrac{n-1}{2}$ is an integer. By \eqref{findmax}, we have 
   $$\depth R/J(P_{n})^t=n-\Delta-1=(N+1)-M-1=N-\Delta=2A-\Delta. $$
   By \eqref{findMM}, it follows
   $$\depth R/J(P_{n})^t=2A-\Delta=2A-mt-\left\lfloor  \dfrac{r+1}{2}\right\rfloor .$$
   Since $r=N-m(2t+1)=2A-m(2t+1)$, one deduces $mt=A-\dfrac{m+r}{2}.$ Therefore, 
   $$\depth R/J(P_{n})^t=A+\dfrac{m+r}{2}-\left\lfloor  \dfrac{r+1}{2}\right\rfloor=A+\dfrac{m}{2}+\bigg(\dfrac{r}{2}-\left\lfloor  \dfrac{r+1}{2}\right\rfloor\bigg).$$
    Set $v=\frac{r}{2}-\left\lfloor  \frac{r+1}{2}\right\rfloor$, we have $v=0$ if $r$ is even, and $v=\dfrac{-1}{2}$ if $r$ is odd. Then 
    $$\depth R/J(P_{n})^t=A+ \frac{m}{2}-v.$$
        We investigate  the following two cases:
   
 {\it Case $1$:} If $m=2p$ is even,  it follows  $r$  has to even. Then $v=\dfrac{-1}{2}$ and thus $\depth R/J(P_{n})^t=A+p+\dfrac{1}{2}=A+\left\lfloor  \frac{m}{2}\right\rfloor.$
 
    {\it Case $2$:} If $m=2p+1$ is odd, then  $\frac{m}{2}=p+\frac{1}{2}$. it follows  $r$  has to odd. Then $v=0$ and thus $\depth R/J(P_{n})^t=A+p=A+\left\lfloor  \frac{m}{2}\right\rfloor.$
   From Cases $1$ and $2$, we can write
   \begin{align}\label{depthodd}
   \depth R/J(P_{n})^t=A+\left\lfloor  \frac{m}{2}\right\rfloor.
   \end{align}
       Because of $m=\left\lfloor  \dfrac{N}{2t+1}\right\rfloor=\left\lfloor  \dfrac{2A}{2t+1}\right\rfloor$, and by Lemma \ref{phannguyen}$(ii)$, we get 
   \begin{align}\label{depthodd-even}
 \left\lfloor  \frac{m}{2}\right\rfloor=\left\lfloor  \frac{\left\lfloor  \dfrac{2A}{2t+1}\right\rfloor}{2}\right\rfloor=\left\lfloor  \dfrac{A}{2t+1}\right\rfloor.
    \end{align}
   Put \eqref{depthodd-even} in \eqref{depthodd}, we have the general formula for all odd paths, as follows:
    \begin{align}\label{aaaa}
   \depth R/J(P_{n})^t=A+\left\lfloor  \dfrac{A}{2t+1}\right\rfloor.
   \end{align}
   Finally, we demonstrate the formula \eqref{aaaa} for each $n=4k+3$, and $n=4k+1.$
   
   {\it Case $(i)$} If $n=4k+3$, $A=\dfrac{n-1}{2}=\dfrac{4k+3-1}{2}=2k+1$. As $0\leqslant t \leqslant k$, then $1\leqslant 2t+1 \leqslant 2k+1$. Thus $\left\lfloor  \dfrac{2k+1}{2t+1}\right\rfloor\geqslant 1,$ and we have
   \begin{align*}
   \depth R/J(P_{n})^t=(2k+1)+\left\lfloor  \dfrac{2k+1}{2t+1}\right\rfloor.
   \end{align*}
   
    {\it Case $(ii)$} If $n=4k+1$, $A=\dfrac{n-1}{2}=\dfrac{4k+1-1}{2}=2k$. As $1\leqslant t \leqslant k-1$, then $3\leqslant 2t+1 \leqslant 2k-1$. Thus 
    $$\dfrac{2k}{2t+1}\geqslant \dfrac{2k}{2k-1}>1.$$ Therefore, $\left\lfloor  \dfrac{2k}{2t+1}\right\rfloor\geqslant 1,$ and  we have
   \begin{align*}
   \depth R/J(P_{n})^t=2k+\left\lfloor  \dfrac{2k}{2t+1}\right\rfloor.
   \end{align*}
   Hence the theorem holds.
      \end{proof}

  \section{Acknowledgment} 
This research was funded by the TNU-University of Sciences for research group: NNC.ĐHKH.2025.03.

\vspace{1cm}
\noindent {\bf Data Availability} Data sharing is not applicable to this article as no datasets were generated or analyzed during the current study.

\noindent {\bf Conflict of interest} There are no competing interests of either financial or personal nature.

\end{document}